\documentclass[a4paper,11pt,twoside]{article}
\usepackage[all]{xy}
\usepackage{graphicx}
\usepackage{pgf,tikz}
\usepackage{mathrsfs}
\usetikzlibrary{arrows}
\usepackage{lipsum}
\usepackage{mathtools}

\definecolor{black}{cmyk}{1.,1.,1.,1.0}
\definecolor{blue}{cmyk}{1.,1.,0.,0.63}
\definecolor{green}{cmyk}{1.,0.,1.,0.63}

\usepackage[backend=bibtex,url=false,eprint=false]{biblatex}
\usepackage{hyperref}
\hypersetup{
       unicode=true,          
     pdftoolbar=true,        
    pdfmenubar=true,        
    pdffitwindow=false,     
    pdfstartview={FitH},    
       pdfproducer={}, 
    pdfcreator={pdfLaTeX},   
    pdfproducer={}, 
    pdfnewwindow=true,      
    colorlinks=false,       
    linkcolor=blue,          
    citecolor=green,        
    filecolor=magenta,      
    urlcolor=cyan           
}
\usepackage[left=2cm,right=2cm,top=2cm,bottom=1.8cm]{geometry}
\usepackage{amssymb,amsmath,amsthm}

\newcommand{\vf}{\vfill\end{document}}

\theoremstyle{definition}

\newtheorem{thm}{Theorem}[section]
\newtheorem{lem}{Lemma}[section]

\numberwithin{equation}{section}
\newcommand{\thistheoremname}{}
\newtheorem*{genericthm*}{\thistheoremname}
\newenvironment{namedthm*}[1]{\renewcommand{\thistheoremname}{#1}%
\begin{genericthm*}}
{\end{genericthm*}}

\newtheoremstyle{named}{}{}{\itshape}{}{\bfseries}{.}{.5em}{\thmnote{#3's }#1}
\theoremstyle{named}

\title{CONSTRUCTION OF HYPERBOLIC HYPERSURFACES
\\
OF LOW DEGREE IN $\mathbb{P}^n(\mathbb{C})$}

\providecommand{\keywords}[1]{\textbf{\textit{Keywords:}} #1}

\author{Dinh Tuan HUYNH}

\newcommand{\Addresses}{{
\bigskip
\footnotesize
\textsc{Dinh Tuan Huynh, Laboratoire de Math\'{e}matiques d'Orsay, Univ. Paris-Sud, CNRS, Universit\'{e} Paris-Saclay, 91405 Orsay, France.}\\
\textsc{Department of Mathematics, College of Education, Hue University, 34 Le Loi St., Hue City, Vietnam.}
\par\nopagebreak
\textit{E-mail address}: \texttt{dinh-tuan.huynh@math.u-psud.fr}}}
\date{\vspace{-5ex}}
\begin{document}
\maketitle
\begin{abstract}
We construct families of hyperbolic hypersurfaces $X_d\subset\mathbb{P}^{n+1}(\mathbb{C})$ of degree $d\geq {\textstyle{(\frac{n+3}{2})^2}}$.
\end{abstract}
, {hyperbolicity}, {Brody Lemma}, {Nevanlinna Theory}

\section{Introduction and the main result}

It was conjectured by Kobayashi \cite{Kobayashi1970} in 1970 that a generic hypersurface $X_d\subset\mathbb{P}^{n+1}(\mathbb{C})$ of sufficiently high degree $d\geq d(n) \gg 1$ is hyperbolic. According to Zaidenberg \cite{Zaidenberg1987}, the optimal degree bound should be $d(n)=2n+1$.

This conjecture, with nonoptimal degree bound in the assumption, was proved, in the case of surface in $\mathbb{P}^3(\mathbb{C})$, by Demailly and El Goul \cite{demailly_goul2000}, and later, by P\v{a}un \cite{mihaipaun2008} with a slight improvement of the degree bound, and in the case of three-fold in $\mathbb{P}^4(\mathbb{C})$ \cite{Rousseau2007}, \cite{Diverio-Trapani2010}. For arbitrary $n$, it was proved in \cite{DMR2010} that any entire curve in generic hypersurface $X_d\subset\mathbb{P}^{n+1}(\mathbb{C})$ of degree $d\geq 2^{n^5}$ must be algebraically degenerate. An improvement of the effective degree bound in this result was given in \cite{Lionel2016}. Recently, for any dimension $n$, a positive answer for generic hypersurfaces of degree $d\geq d(n)\gg 1$ very high was proposed by Siu \cite{siu2015}, and a strategy which is expected to give a confirmation of this conjecture for {\em very} generic hypersurfaces of degree $d\geq 2n+2$ was announced by Demailly \cite{demailly2015}.

Another direction on this subject is to construct examples of hyperbolic hypersurfaces of low degree. In low dimensional case, several examples of hyperbolic hypersurfaces were given. The first example of a hyperbolic surface in $\mathbb{P}^3(\mathbb{C})$ was constructed by Brody and Green \cite{Brody-Green1977}. In $\mathbb{P}^3(\mathbb{C})$, Duval \cite{duval2004} gave an example of a hyperbolic surface of degree $6$, which is the lowest degree found up to date. Later, Ciliberto and Zaidenberg \cite{Ciliberto_Zaidenberg2013} gave a new construction of hyperbolic surface of degree $6$ and their method works for all degree $d\geq 6$ (hence, this is the first time when a hyperbolic surface of degree $7$ was created). In \cite{Huynh2015}, we constructed families of hyperbolic hypersurfaces of degree $d=d(n)=2n+2$ for $2\leq n\leq 5$ (the method works for all $d\geq 2n+2$). The first examples in any dimension $n\geq 4$ were discovered by Masuda and Noguchi \cite{masuda_noguchi1996}, with high degree. Improving this result, examples of hyperbolic hypersurfaces with lower degree asymptotic were given by Siu and Yeung \cite{siu_yeung1997} with $d(n)=16\,n^2$, and by Shiffman and Zaidenberg \cite{shiffman_zaidenberg2002_pn} with $d(n)=4\,n^2$.

In this note, using the technique of \cite{Huynh2015}, we improve the result of Shiffman and Zaidenberg \cite{shiffman_zaidenberg2002_pn} by proving that a small deformation of a union of $q\geq {\textstyle{(\frac{n+3}{2})^2}}$ hyperplanes in general position in $\mathbb{P}^{n+1}(\mathbb{C})$ is hyperbolic.

A family of hyperplanes $\{H_i\}_{1\leq i\leq q}$ with $q\geq n+1$ in $\mathbb{P}^n(\mathbb{C})$ is said to be in \textsl{general position} if any $n+1$ hyperplanes in this family have empty intersection, namely if
\[
\cap_{i\in I}H_i=\emptyset,\quad\forall\, I\subset \{1,\ldots, q\},\, |I|=n+1.
\]

Let $\{H_i\}_{1\leq i\leq q}$ be a family of hyperplanes in general position in $\mathbb{P}^n(\mathbb{C})$. A hypersurface $S$ in $\mathbb{P}^n(\mathbb{C})$ is said to be in \textsl{general position with respect to $\{H_i\}_{1\leq i\leq q}$} if it avoids all intersection points of $n$ hyperplanes, namely if
\[
S\cap\big(\cap_{i\in I}H_i\big)=\emptyset,\quad\forall \, I\subset \{1,\ldots, q\}, \, |I|=n.
\]

\begin{namedthm*}{Main Theorem}\label{maintheorem}
Let $\{H_i\}_{1\leq i\leq q}$  be a family of $q\geq {\textstyle{(\frac{n+3}{2})^2}}$ hyperplanes in general position in $\mathbb{P}^{n+1}(\mathbb{C})$, where $H_i=\{h_i=0\}$. Then there exists a hypersurface $S=\{s=0\}$ of degree $q$ in general position with respect to $\{H_i\}_{1\leq i\leq q}$ such that the hypersurface 
\[
\Sigma_{\epsilon}
=
\big\{\epsilon s+\Pi_{i=1}^{q}h_i=0\big\}
\]
is hyperbolic for sufficiently small complex $\epsilon\not=0$.
\end{namedthm*}

\section*{Acknowledgments}
I would like to gratefully thank Julien Duval for turning my attention to this problem and for his inspiring discussions on the subject. I am specially thankful to Jo\"{e}l Merker for his encouragements and his comments that greatly improved the manuscript.

\section{Preparations}
\subsection{Brody Lemma and its applications}

Let $X$ be a compact complex manifold equipped with a hermitian metric $\| \cdot\|$. An \textsl{entire curve} in $X$ is a nonconstant holomorphic map $f:\mathbb{C}\rightarrow X$. Such an $f:\mathbb{C}\rightarrow X$ is called a \textsl{Brody curve} if its derivative $\| f'\|$ is bounded. The following result \cite{brody1978} is a useful tool for studying complex hyperbolicity.

\begin{namedthm*}{Brody Lemma}
\label{brodylemma}
Let $f_k:\mathbb{D}\rightarrow X$ be a sequence of holomorphic maps from the unit disk to a compact complex manifold $X$. If $\|f_k'(0)\|\rightarrow\infty$ as $k\rightarrow \infty$, then there exist a point $a\in\mathbb{D}$, a sequence $(a_k)$ converging to $a$ and a decreasing sequence $(r_k)$ of positive real numbers converging to 0 such that the sequence of maps
\[
z\rightarrow f_k(a_k+ r_k\, z)
\]
converges toward a Brody curve, after extracting a subsequence.
\end{namedthm*}

Consequently, we have a well-known characterization of Kobayashi hyperbolicity.

\begin{namedthm*}{Brody Criterion}\label{brody_criterion}
A compact complex manifold $X$ is Kobayashi hyperbolic if and only if it contains no entire curve.
\end{namedthm*}

The following form of the Brody Lemma shall be repeatedly used in the proof of the Main Theorem.

\begin{namedthm*}{Sequences of entire curves}\label{lemma brody_use}

Let $X$ be a compact complex manifold and let $(f_k)$ be a sequence of entire curves in $X$. Then there exist a sequence of reparameterizations $r_k:\mathbb{C}\rightarrow\mathbb{C}$ and a subsequence of $(f_k\circ r_k)$ which converges toward an entire curve.
\end{namedthm*}

\subsection{Stability of intersections}

We recall here the following known complex analysis fact.

\begin{namedthm*}{Stability of intersections}\label{stability}
 Let $X$ be a complex manifold and let $H \subset X$ be an analytic hypersurface. Suppose that a sequence $(f_k)$ of entire curves in $X$ converges toward an entire curve $f$. If $f(\mathbb{C})$ is not contained in $H$, then
\[
f(\mathbb{C})
\cap H
\subset
\lim f_k(\mathbb{C})\cap H.
\]
\end{namedthm*}

\subsection{Hyperbolicity of the complement of $2n+1$ hyperplanes in general position in $\mathbb{P}^n(\mathbb{C})$}

We also need the classical generalization of Picard's theorem (case $n=1$) \cite{Fujimoto1972}.

\begin{thm}
\label{fujimoto-green}
The complement of a collection of $2n+1$ hyperplanes in general position in $\mathbb{P}^n(\mathbb{C})$ is hyperbolic.
\end{thm}

\section{Proof of the Main Theorem}

Given a hypersurface $S$ of degree $q$ in general position with respect to the family $\{H_i\}_{1\leq i\leq q}$, we would like to determine what conditions $S$ should satisfy for $\Sigma_{\epsilon}$ to be hyperbolic. Suppose that $\Sigma_{\epsilon_k}$ is not hyperbolic for a sequence $(\epsilon_k)$ converging to $0$. Then we can find entire curves $f_{\epsilon_k}\colon\mathbb{C}\rightarrow\Sigma_{\epsilon_k}$. By the Brody Lemma, after reparametrization and extraction, we may assume that the sequence $(f_{\epsilon_k})$ converges to an entire curve $f\colon\mathbb{C}\rightarrow \cup_{i=1}^{q}H_i$. By uniqueness principle, the curve $f(\mathbb{C})$ lands in $\cap_{i\in I}H_i$, for some subset $I$ of the index set $\mathbf{Q}:=\{1,\ldots,q\}$ and does not land in any $H_j$ with $j\in\mathbf{Q}\setminus I$.

\begin{lem}
One has
\[
|I|\leq n-1.
\]
\end{lem}

\begin{proof}
If on the contrary $|I|=n$, then for all $j\in \mathbf{Q}\setminus I$, by stability of intersections, one has
\[
f(\mathbb{C})\cap H_j
\subset
\lim f_{\epsilon_k}(\mathbb{C})\cap H_j\subset
\lim\Sigma_{\epsilon_k}\cap H_j\subset S\cap H_j.
\]
Thus, $f(\mathbb{C})\cap H_j\subset S\cap H_j\cap\big(\cap_{i\in I}H_i\big)=\emptyset$. Hence, $f(\mathbb{C})\subset
\cap_{i\in I}H_i\setminus\big(\cup_{j\in \mathbf{Q}\setminus I} H_j\big)$, which is a contradiction, since the complement of $q-|I|>3$ points in a line is hyperbolic by Picard's theorem.
\end{proof}

By the above argument, $f(\mathbb{C})\cap H_j$ is contained in $S$ for all $j\in \mathbf{Q}\setminus I$. Therefore, the curve $f(\mathbb{C})$ lands in
\begin{equation}\label{reduced_problem}
\cap_{i\in I}H_i\setminus\big(\cup_{j\in \mathbf{Q}\setminus I}H_j\setminus S\big).
\end{equation}
So, the problem reduces to finding a hypersurface $S$ of degree $q$ such that all complements of the form \eqref{reduced_problem} are hyperbolic, where $I$ is an arbitrary subset of $\mathbf{Q}$ having cardinality at most $n-1$.

Such a hypersurface $S$ will be constructed by using the deformation method of Zaidenberg and Shiffman \cite{Shiffman_Zaidenberg2005}.

\textit{Starting point of the deformation process.} Let $\{H_i\}_{1\leq i\leq q}$ be a family of hyperplanes in general position in $\mathbb{P}^n(\mathbb{C})$. For some integer $0\leq k\leq n-1$ and some subset $I_k=\{i_1,\dots,i_{n-k}\}$ of the index set $\{1,\ldots,q\}$ having cardinality $n-k$, the linear subspace $P_{k,I_k}= \cap_{i\in I_k}H_i\simeq\mathbb{P}^k(\mathbb{C})$ will be called a \textsl{subspace of dimension $k$}. We will denote by $P_{k,I_k}^*$ the complement $P_{k,I_k}\setminus\big(\cup_{i\not\in I_k} H_i\big)$, which we will call a \textsl{star-subspace of dimension $k$}. The process of constructing $S$ by deformation will start with the following result, which is an application of Theorem~\ref{fujimoto-green}.

\begin{namedthm*}{Starting Lemma}
\label{starting}
Let $\{H_i\}_{1\leq i\leq q}$ be a family of $q\geq {\textstyle{(\frac{n+3}{2})^2}}$ hyperplanes in general position in $\mathbb{P}^{n+1}(\mathbb{C})$. Let $I$ and $J$ be two disjoint subsets of the index set $\{1,\ldots,q\}$ such that $1\leq |I|\leq n-1$, and $|J|= q+m+1-2|I|$ with some $0\leq m\leq |I|-1$. Then all complements of the form
\begin{equation}\label{starting_form}
\cap_{i\in I}H_i\setminus\big(\cup_{j\in J}H_j\setminus A_{m,n+1-|I|}\big)
\end{equation}
are hyperbolic, where $A_{m,n+1-|I|}$ is a set of at most $m$ star-subspaces coming from the family of hyperplanes $\{\cap_{i\in I}H_i\cap H_j\}_{j\in J}$ in the $(n+1-|I|)$--dimensional projective space $\cap_{i\in I}H_i\cong\mathbb{P}^{n+1-|I|}(\mathbb{C})$.
\end{namedthm*}

\begin{proof}
Suppose on the contrary that there exists an entire curve $f:\mathbb{C}\rightarrow\cap_{i\in I}H_i\setminus\big(\cup_{j\in J}H_j\setminus A_{m,n+1-|I|}\big)$. Since each star-subspace in $A_{m,n+1-|I|}$ is constructed from at most $n+1-|I|$ hyperplanes in the family $\{\cap_{i\in I}H_i\cap H_j\}_{j\in J}$, the curve $f$ must avoid completely at least $|J|-m(n+1-|I|)$ hyperplanes in the projective space $\cap_{i\in I}H_i\cong\mathbb{P}^{n+1-|I|}(\mathbb{C})$. By the elementary estimate
\begin{align*}
|J|-m(n+1-|I|)
&
=
q+1-2|I|-m(n-|I|)\\
&
\geq
2(n+1-|I|)+1+\bigg[\bigg(\dfrac{n+3}{2}\bigg)^2-2(n+1)-(|I|-1)(n-|I|)\bigg]\\
&
\geq
2(n+1-|I|)+1,
\end{align*}
and by using Theorem~\ref{fujimoto-green}, we derive a contradiction.
\end{proof}

\textit{Deformation lemma.} For $2\leq l\leq n$, let $\Delta_l$ be a finite collection of subspaces of dimension $n+1-l$ coming from the family $\{H_i\}_{1\leq i\leq q}$, possibly with $\Delta_l=\emptyset$, and let $D_l\not \in \Delta_l$ be another subspace of dimension $n+1-l$, defined as $D_l=\cap_{i\in I_{D_l}}H_i$. For an arbitrary hypersurface $S=\{s=0\}$ in general position with respect to the family $\{H_i\}_{1\leq i\leq q}$ and for $\epsilon\not=0$, we set
\[
S_{\epsilon}
=
\big\{\epsilon s+\Pi_{i\not\in I_{D_l}}h_i^{n_i}=0\big\},
\]
where $n_i\geq 1$ are chosen (freely) so that $\sum_{i\not\in I_{D_l}}n_i=q$. Then the hypersurface $S_{\epsilon}$ is also in general position with respect to $\{H_i\}_{1\leq i\leq q}$. We denote by $\overline{\Delta}_l$ the family of all subspaces of dimension $n+1-l$ ($2\leq l\leq n+1$), with the convention $\overline{\Delta}_{n+1}=\emptyset$. We shall apply inductively the following lemma.

\begin{lem}\label{lem_deformation}
Assume that all complements of the form
\begin{equation}\label{start_form}
\cap_{i\in I}H_i\setminus\big(\cup_{j\in J}H_j\setminus(((\Delta_l\cup\overline{\Delta}_{l+1})\cap S)\cup A_{m,n+1-|I|})\big)
\end{equation}
are hyperbolic where $I$ and $J$ are two disjoint subsets of the index set $\{1,\ldots,q\}$ such that $1\leq |I|\leq n-1$, and $|J|= q+m+1-2|I|$ with some $0\leq m\leq |I|-1$, and where $A_{m,n+1-|I|}$ is a set of at most $m$ star-subspaces coming from the family of hyperplanes $\{\cap_{i\in I}H_i\cap H_j\}_{j\in J}$ in $\cap_{i\in I}H_i\cong\mathbb{P}^{n+1-|I|}(\mathbb{C})$. Then all complements of the form
\begin{equation}\label{end_form}
\cap_{i\in I}H_i\setminus\big(\cup_{j\in J}H_j\setminus(((\Delta_l\cup D_l\cup\overline{\Delta}_{l+1})\cap S_{\epsilon})\cup A_{m,n+1-|I|})\big)
\end{equation}
are also hyperbolic for sufficiently small $\epsilon\not=0$.
\end{lem}

\begin{proof}
By the definition of $S_{\epsilon}$, we see that $S_{\epsilon}\cap\big(\cap_{m\in M}H_m\big)= S\cap\big(\cap_{m\in M}H_m\big)$ when $M\cap(\mathbf{Q}\setminus I_{D_l})\not=\emptyset$, hence
\[
(\Delta_l\cup D_l\cup\overline{\Delta}_{l+1})\cap S_{\epsilon}=
((\Delta_l\cup \overline{\Delta}_{l+1})\cap S)\cup(D_l\cap S_{\epsilon}).
\]
When $|I|\geq l$, using this, we observe that the two complements \eqref{start_form}, \eqref{end_form} coincide.

Assume therefore $|I|\leq l-1$. Suppose by contradiction that there exists a sequence of entire curves $f_{\epsilon_k}(\mathbb{C})$, $\epsilon_k\rightarrow 0$, contained in the complement \eqref{end_form} for $\epsilon=\epsilon_k$. By the Brody Lemma, we may assume that $(f_{\epsilon_k})$ converges to an entire curve $f(\mathbb{C})\subset \cap_{i\in I}H_i$. We are going to prove that the curve $f(\mathbb{C})$ lands in some complement of the form \eqref{start_form}.

Let $\cap_{k\in K}H_k$ be the smallest subspace containing $f(\mathbb{C})$, so that $I$ is a subset of $K$. Take an index $j$ in $J\setminus K$. By stability of intersections, we have
\begin{align}
\label{intersection}
f(\mathbb{C})\cap H_j& \subset \lim f_{\epsilon_k}(\mathbb{C})\cap H_j\notag\\
&\subset ((\Delta_l\cup \overline{\Delta}_{l+1})\cap S)\cup A_{m,n+1-|I|}\cup\lim(D_l\cap S_{\epsilon_k}).
\end{align}
If the index $j$ does not belong to $I_{D_l}$, then $H_j\cap D_l\cap S_{\epsilon_k}\subset \overline{\Delta}_{l+1}\cap S$. It follows from \eqref{intersection} that
\begin{equation}
\label{intersection1}
f(\mathbb{C})\cap H_j
\subset ((\Delta_l\cup \overline{\Delta}_{l+1})\cap S)\cup A_{m,n+1-|I|}.
\end{equation}
If the index $j$ belongs to $I_{D_l}$, noting that $\lim (D_l\cap S_{\epsilon_k})$ is contained in $D_l\cap (\cup_{i\not\in I_{D_l}} H_i)$, hence from \eqref{intersection}
\begin{equation}
\label{intersection2}
f(\mathbb{C})\cap H_j
\subset
((\Delta_l\cup \overline{\Delta}_{l+1})\cap S)\cup A_{m,n+1-|I|}\cup (D_l\cap(\cup_{i\not\in I_{D_l}}H_i)).
\end{equation}

Assume first that $K=I$. We claim that \eqref{intersection1} also holds when the index $j\in J\setminus I$ belongs to $I_{D_l}$. Indeed, for the supplementary part in \eqref{intersection2}, we have
\[
f(\mathbb{C})\cap H_j\cap
\big(
D_l\cup_{i\not\in I_{D_l}}
H_i
\big)
\subset
\cup_{i\not\in I_{D_l}}
(f(\mathbb{C})\cap H_j\cap H_i),
\]
so that \eqref{intersection1} applies here to all $i\not\in I_{D_l}$. Hence, the curve $f(\mathbb{C})$ lands inside
\[
\cap_{i\in I}H_i
\setminus
\big(
\cup_{j\in J}H_j\setminus(((\Delta_l\cup\overline{\Delta}_{l+1})\cap S)\cup A_{m,n+1-|I|})\big),
\]
contradicting the hypothesis.\\

Assume now that $I$ is a proper subset of $K$. Let us set
\[
A_{m,n+1-|I|,K}
=
\{X\cap(\cap_{k\in K}H_k)|X\in A_{m,n+1-|I|}\}.
\]
This set consists of star-subspaces of $\cap_{k\in K}H_k\cong\mathbb{P}^{n+1-|K|}(\mathbb{C})$. Let $B_{m,K}$ be the subset of $A_{m,n+1-|I|,K}$ containing all star-subspaces of dimension $n-|K|$ (i.e. of codimension $1$ in $\cap_{k\in K}H_k$), and let $C_{m,K}$ be the remaining part. A star-subspace in $B_{m,K}$ is of the form $(\cap_{k\in K}H_k\cap H_j)^*$ for some index $j\in J\setminus K$. Let then $R$ denote the set of such indices $j$, so that
\[
|R|=|B_{m,K}|.
\]
We consider two cases separately, depending on the dimension of the subspace $Y=\cap_{k\in K}H_k\cap D_l$.

\begin{itemize}
\item[\textbf{Case 1:}] $Y$ is a subspace of dimension $n-|K|$. In this case, $Y$ is of the form $(\cap_{k\in K}H_k)\cap H_y$ for some index $y$ in $I_{D_l}$. It follows from \eqref{intersection}, \eqref{intersection1}, \eqref{intersection2} that the curve $f(\mathbb{C})$ lands inside 
\[
\cap_{k\in K}H_k
\setminus\big(\cup_{j\in (J\setminus K)\setminus (R\cup\{y\})}H_j\setminus(((\Delta_l\cup\overline{\Delta}_{l+1})\cap S)\cup C_{m,K})\big).
\]
To conclude that this set is of the form \eqref{start_form}, we need to show that
\begin{itemize}
\item[\textbf{(1)}] $|(J\setminus K)\setminus (R\cup\{y\})|=q+m'+1-2|K|$ with $|C_{m,K}|\leq m'\leq |K|-1$;
\item[\textbf{(2)}] $|K|\leq n-1$.
\end{itemize}

Consider \textbf{(1)}. We need to verify the corresponding
required inequality between cardinalities
\[
|C_{m,K}|\leq |(J\setminus K)\setminus (R\cup\{y\})|-q+2|K|-1\leq |K|-1.
\]
The right inequality is equivalent to
\[
|(J\setminus K)\setminus (R\cup\{y\})|
\leq
|\{1,\ldots,q\}\setminus K|,
\]
which is trivial. The left inequality follows from the elementary estimates
\begin{align*}
|(J\setminus K)\setminus (R\cup\{y\})|-q+2|K|-1
&
\,
\geq
\,
|J\setminus K|-|B_{m,K}|-q+2|K|-2\\
&
\,
=
\,
|J|-|J\cap K|-|B_{m,K}|-q+2|K|-2\\
&
\,
=
\,
(m-|B_{m,K}|)+(2|K|-2|I|-|J\cap K|-1)
\\
&
\,
\geq
|C_{m,K}|,
\end{align*}
where the last inequality holds because $I$ and $J$ are two disjoint sets and $I$ is a proper subset of $K$.

Consider \textbf{(2)}. Suppose on the contrary that $|K|=n$. Since $S$ is in general position with respect to $\{H_i\}_{1\leq i\leq 2n+2}$, we see that
\[
\cap_{k\in K}H_k
\setminus\big(\cup_{j\in (J\setminus K)\setminus (R\cup\{y\})}H_j\setminus(((\Delta_l\cup\overline{\Delta}_{l+1})\cap S)\cup C_{m,K})\big)
=
\cap_{k\in K}H_k
\setminus\big(\cup_{j\in (J\setminus K)\setminus (R\cup\{y\})}H_j\setminus C_{m,K}\big).
\]
Since $|(J\setminus K)\setminus (R\cup\{y\})|\geq q+1-2n+|C_{m,K}|\geq 3+|C_{m,K}|$, the curve $f$ lands in a complement of at least $3$ points in a line. By Picard's Theorem, $f$ is constant, which is a contradiction.
\item[\textbf{Case 2:}] $Y$ is a subspace of dimension at most $n-|K|-1$. In this case, the curve $f(\mathbb{C})$ lands inside
\[
\cap_{k\in K}H_k
\setminus
\big(
\cup_{j\in (J\setminus K)\setminus R}H_j
\setminus
(((\Delta_l\cup\overline{\Delta}_{l+1})\cap S)\cup C_{m,K}\cup Y^*)
\big),
\]
which is also of the form \eqref{start_form}, since
\[
|(J\setminus K)\setminus R|
\,
\geq
\,
q-2|K|+1+|C_{m,K}\cup Y^*|,
\]
and since $|K|\leq n-1$, by similar arguments as in \textbf{Case 1}.
\end{itemize}

The Lemma is thus proved.
\end{proof}

\begin{proof}[Inductive deformation process and end of the proof of the Main Theorem]
We may begin by applying Lemma~\ref{lem_deformation} for $l=n$ (with $\overline{\Delta}_{n+1}=\emptyset$), firstly with $\Delta_n=\emptyset$, and with some $D_n\in\overline{\Delta}_n$, since $(\Delta_n\cup\overline{\Delta}_{n+1})\cap S=\emptyset$, hence the assumption of this lemma holds by the Starting Lemma. Next, we reapply Lemma~\ref{lem_deformation} inductively until we exhaust all $D_n\in\overline{\Delta}_{n}$. We get at the end a hypersurface $S_1$ such that all complements of the forms
\begin{align*}
&\cap_{i\in I}H_i
\setminus
\big(
\cup_{j\in J}H_j\setminus(S_1\cup A_{m,n+1-|I|})
\big)\quad\quad\quad\quad\quad\quad\, {\scriptstyle(|I|\,=\,n-1)}\\
&\cap_{i\in I}H_i
\setminus
\big(
\cup_{j\in J}H_j\setminus((\overline{\Delta}_{n}\cap S_1)\cup A_{m,n+1-|I|})
\big)\quad\quad\quad {\scriptstyle(|I|\,\leq\, n-2)}
\end{align*}
are hyperbolic, since when $|I|=n-1$, two components $
\cap_{i\in I}H_i
\setminus
\big(
\cup_{j\in J}H_j\setminus((\overline{\Delta}_{n}\cap S_1)\cup A_{m,n+1-|I|})
\big)$ and $
\cap_{i\in I}H_i
\setminus
\big(
\cup_{j\in J}H_j\setminus(S_1\cup A_{m,n+1-|I|})
\big)$ are equal.
Considering this as the starting point of the second step, we apply inductively Lemma \ref{lem_deformation} for $l=n-1$ and receive at the end a hypersurface $S_2$ such that all complements of the forms
\begin{align*}
&\cap_{i\in I}H_i
\setminus
\big(
\cup_{j\in J}H_j\setminus(S_2\cup A_{m,n+1-|I|})
\big)\quad\quad\quad\quad\quad\quad\, {\scriptstyle(n-2\,\leq\,|I|\,\leq\,n-1)}\\
&\cap_{i\in I}H_i
\setminus
\big(
\cup_{j\in J}H_j\setminus((\overline{\Delta}_{n-1}\cap S_2)\cup A_{m,n+1-|I|})
\big)\quad\quad {\scriptstyle(|I|\,\leq\, n-3)}
\end{align*}
are hyperbolic, for the same reason as in above. Continuing this process, we get at the end of the $(n-1)^{\text{th}}$ step a hypersurface $S=S_{n-1}$ such that all complements of the forms
\[
\cap_{i\in I}H_i
\setminus
\big(
\cup_{j\in J}H_j\setminus(S_{n-1}\cup A_{m,n+1-|I|})
\big)\quad\quad\quad\quad\, {\scriptstyle(1\,\leq\,|I|\,\leq\,n-1)}
\]
are hyperbolic. In particularly, by choosing $m=|I|-1$, whence $|J|=q-|I|$, and by choosing $A_{m,n+1-|I|}=\emptyset$, all complements of the form \eqref{reduced_problem} are hyperbolic for $S=S_{n-1}$.
\end{proof}
\centering
\printbibliography
\Addresses
\end{document}